\documentclass[11pt,reqno]{amsart}
\usepackage{indentfirst, amssymb,amsmath,amsthm, mathrsfs,cite,graphicx,float}
\newtheorem*{theoA}{Theorem A}
\newtheorem*{theoB}{Theorem B}
\newtheorem*{theoC}{Theorem C}
\newtheorem*{theoD}{Theorem D}

\newtheorem*{theoF}{Theorem F}
\newtheorem*{theoG}{Theorem G}
\newtheorem*{theoH}{Theorem H}

\newtheorem*{theoJ}{Theorem J}

\newtheorem{theo}{Theorem}[section]
\newtheorem{lem}{Lemma}[section]

\newtheorem{cor}{Corollary}[section]
\newtheorem*{corA}{Corollary A}

\newtheorem{open problem}{Open problem}[section]
\newcommand{\pa}{\partial}
\newcommand{\ol}{\overline}
\newcommand{\be}{\begin{equation}}
\newcommand{\ee}{\end{equation}}
\newcommand{\bs}{\begin{small}}
\newcommand{\es}{\end{small}}
\newcommand{\beas}{\begin{eqnarray*}}
\newcommand{\eeas}{\end{eqnarray*}}
\newcommand{\bea}{\begin{eqnarray}}
\newcommand{\eea}{\end{eqnarray}}
\renewcommand{\epsilon}{\varepsilon}
\numberwithin{equation}{section}

\begin{document}
\title[Bohr phenomenon]{Bohr phenomenon for analytic and harmonic mappings on shifted disks}
\author[V. Allu, R. Biswas and R. Mandal]{ Vasudevarao Allu, Raju Biswas and Rajib Mandal}
\date{}
\address{Vasudevarao Allu, Department of Mathematics, School of Basic Sciences, Indian Institute of Technology Bhubaneswar, Bhubaneswar-752050, Odisha, India}
\email{avrao@iitbbs.ac.in}
\address{Raju Biswas, Department of Mathematics, Raiganj University, Raiganj, West Bengal-733134, India.}
\email{rajubiswasjanu02@gmail.com}
\address{Rajib Mandal, Department of Mathematics, Raiganj University, Raiganj, West Bengal-733134, India.}
\email{rajibmathresearch@gmail.com}
\let\thefootnote\relax
\footnotetext{2020 Mathematics Subject Classification: 30A10, 30H05, 30C35, 30C45, 30C50, 30C80.}
\footnotetext{Key words and phrases: Simply connected domain, bounded analytic functions, harmonic mappings, Bohr radius, improved Bohr radius, refined Bohr radius}
\begin{abstract} The primary objective of this paper is to establish several sharp results concerning the Bohr inequality, the refined Bohr inequality, and the improved Bohr inequality for the classes of analytic functions and harmonic mappings defined on the shifted disks 
\beas \Omega_{\gamma}=\left\{z\in\mathbb{C}:\left|z+\frac{\gamma}{1-\gamma}\right|<\frac{1}{1-\gamma}\right\}\quad\text{for}\quad\gamma\in[0,1).\eeas
\end{abstract}
\maketitle
\section{Introduction and Preliminaries}
Let $\mathbb{D}_{\rho}(a)=\{z:|z-a|<\rho\}$ and $\mathbb{D}_1(0):=\mathbb{D}$ the open unit disk in the complex plane $\mathbb{C}$. 
Let $\Omega$ be a simply connected domain, with $\mathbb{D}\subseteq \Omega$ and $\mathcal{H}(\Omega)$ denoting the class of analytic functions on $\Omega$.
Let $\mathcal{B}(\Omega)=\left\{f \in \mathcal{H}(\Omega): f(\Omega)\subseteq \ol{\mathbb{D}}\right\}$, and the Bohr radius (see \cite{10}) for the class $\mathcal{B}(\Omega)$ is defined to be the number $B_\Omega\in (0, 1)$ such 
that 
\beas B_\Omega=\sup\left\{\rho\in (0,1):M_f(\rho)\leq 1\;\text{for} \;f(z)=\sum_{n=0}^{\infty}\alpha_nz^n\in\mathcal{B}(\Omega),z\in\mathbb{D}\right\},\eeas 
where $M_f(\rho)=\sum_{n=0}^{\infty}|\alpha_n|\rho^n$
is the majorant series corresponding to the analytic functions $f\in\mathcal{B}(\Omega)$ in $\mathbb{D}$. It is well known that when $\Omega=\mathbb{D}$, $B_{\mathbb{D}} = 1/3$. This is described as follows:
\begin{theoA} (The Classical Bohr Theorem) If $f\in\mathcal{B}(\mathbb{D})$, then $M_f (\rho) \leq 1$ for $0 \leq \rho \leq 1/3$. The number $1/3$ is best possible.
\end{theoA}
It should be noted that for $\rho > 1/3$, the inequality $M_f (\rho) \leq 1$, where $f\in\mathcal{B}(\mathbb{D})$, does not hold.
To check this, we consider the function $\varphi_r(z)=(r-z)/(1-rz)$, $r\in[0,1)$. It is evident that $M_{\varphi_r}(\rho)>1$ if, and only if, $\rho>1/(1+2r)$.
This shows that $1/3$ is optimal for $r\to1^-$.\\[2mm]
\indent Actually, Bohr \cite{5} proved \textrm{Theorem A} for $\rho \leq 1/6$.
The optimal value $1/3$, which is known as the Bohr radius, was subsequently established as an independent result by Riesz, Schur, and Wiener \cite{5}. 
This outcome was also established by Sidon \cite{26} and Tomi\'c \cite{27}. For an in-depth study on the Bohr radius, we refer to \cite{1,2,305,302,303,301,210,211,212,304,3,6,7,8,9,10,11,306,307,300,13,14,15, 16,309,20,21,23,24,R1,ABM2025, ABM2024, BM2025} and its references.
Using the classical Bohr inequality, Dixon \cite{200} construct a Banach algebra that is not classified as an operator algebra but nevertheless satisfies the non-unital von Neumann's
 inequality. 
 Boas and Khavinson \cite{4} further developed the concept of Bohr radius, especially in the context of several complex variables, and introduced the notion of a 
multidimensional Bohr radius. Several researchers have followed up and extended this phenomenon in various settings, as shown in \cite{213,214,215}.\\[2mm]
For $\gamma\in[0,1)$, we consider the open disk $\Omega_{\gamma}$ defined by
\beas \Omega_{\gamma}=\left\{z\in\mathbb{C}:\left|z+\frac{\gamma}{1-\gamma}\right|<\frac{1}{1-\gamma}\right\}.\eeas
Note that, $\mathbb{D}\subseteq\Omega_{\gamma}$ for all $\gamma\in[0,1)$. 
Figure \ref{fig1} shows pictures of circles $C_\gamma : \left|z+\frac{\gamma}{1-\gamma}\right|=\frac{1}{1-\gamma}$ for different values of $\gamma \in [0, 1)$.
\begin{figure}[H]
\centering
\includegraphics[scale=1]{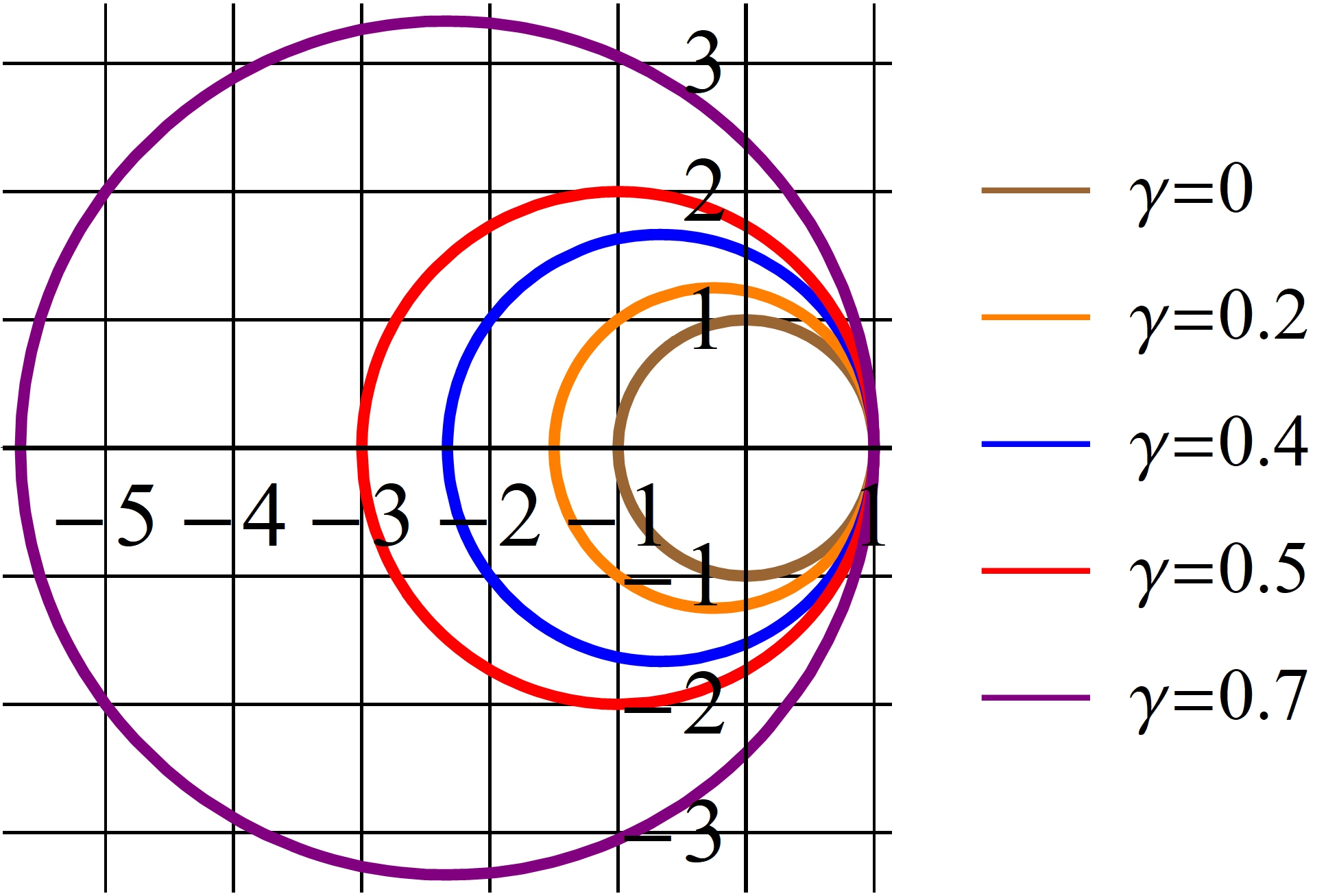}
\caption{The graphs of $C_{\gamma}$ when $\gamma=0,0.2,0.4,0.5,0.7$}
\label{fig1}
\end{figure}
In 2010, Fournier and Ruscheweyh \cite{10} extended the concept of Bohr's inequality as follows:
\begin{theoB}\cite[\textrm{Theorem 1}]{10} For $\gamma\in[0,1)$, let $f \in\mathcal{B}(\Omega_{\gamma})$ with $f(z)=\sum_{n=0}^{\infty} \alpha_nz^n$ for $z\in\mathbb{D}$. Then,
\beas \sum_{n=0}^{\infty} |\alpha_n|\rho^n\leq 1\;\;\text{ for}\;\; \rho \leq \rho_{\gamma} := \frac{1 + \gamma}{3 + \gamma} .\eeas
Also $\sum_{n=0}^{\infty} |\alpha_n|\rho_{\gamma}^n=1$ holds for a function $f(z)=\sum_{n=0}^{\infty} \alpha_nz^n$ in $\mathcal{B}(\Omega_{\gamma})$ if, and only if, $f(z)=c$ with $|c| = 1$.\end{theoB}
In $2021$, Evdoridis {\it et al.}\cite{9a} obtained an improved version of \textrm{Theorem B} as follows:
\begin{theoC}\cite[Theorem 1]{9a} For $\gamma\in[0,1)$, let $f \in\mathcal{B}(\Omega_{\gamma})$ with $f(z)=\sum_{n=0}^{\infty} \alpha_nz^n$ for $z\in\mathbb{D}$. Then,
\beas\sum_{n=0}^{\infty} |\alpha_n|\rho^n+\frac{8}{9}\left(\frac{S_{\rho(1-\gamma)}}{\pi}\right)\leq 1\;\;\text{ for}\;\; \rho \leq \rho_{\gamma} := \frac{1 + \gamma}{3 + \gamma}, \eeas
where the area of the image of the disk $\mathbb{D}(0; \rho)$ under the mapping $f$ is denoted by $S_\rho$. Also, the inequality result is strict unless $f$ is a constant function. The numbers $8/9$ and $(1 + \gamma)/(3 + \gamma)$ cannot be replaced by a larger value.
\end{theoC}
For an analytic function $f(z)=\sum_{n=0}^{\infty} \alpha_nz^n$ for $z\in\mathbb{D}$, we write
\beas \left\Vert f_0\right\Vert_\rho=\sum_{n=1}^{\infty} |\alpha_n|^2\rho^{2n},\quad\text{where}\quad f_0(z)=f(z)-f(0).\eeas 
In 2021, Evdoridis {\it et al}. {\cite{9a}} established the following result, which is a refinement of \textrm{Theorem B} for the class $\mathcal{B}(\Omega_{\gamma})$ with its restriction to the unit disk $\mathbb{D}$.
\begin{theoD}\cite[Theorem 2]{9a} For $\gamma\in[0,1)$, let $f \in\mathcal{B}(\Omega_{\gamma})$ with $f(z)=\sum_{n=0}^{\infty} \alpha_nz^n$ for $z\in\mathbb{D}$. Then,
\beas \sum_{n=0}^\infty |\alpha_n|\rho^n+\left(\frac{1}{1+|\alpha_0|}+\frac{\rho}{1-\rho}\right) ||f_0||_\rho\leq 1\;\;\text{ for}\;\; \rho \leq \rho_0:= \rho_{\gamma} := \frac{1 + \gamma}{3 + \gamma}, \eeas
and the number $\rho_0$ cannot be improved.\end{theoD}
Let $f = u + iv$ be a complex-valued function in a simply connected domain $\Omega$. 
If $f$ satisfies the Laplace equation $\Delta f = 4f_{z\ol z} = 0$, then it is harmonic in $\Omega$, {\it i.e.}, $u$ and $v$ are real harmonic in $\Omega$. 
It should be noted that the canonical representation of every harmonic mapping $f$ is $f = h + \ol{g}$, where $h$ and $g$ are analytic in $\Omega$. This representation is unique 
up to an additive constant (see \cite{20P}). 
The inverse function theorem, as well as a result established by Lewy \cite{18}, indicates that a harmonic function $f$ is locally univalent within the domain $\Omega$ if, and only if, 
the Jacobian of f, defined by $J_f(z)=|h'(z) |^2-|g'(z) |^2$, retains a non-zero value in $\Omega$.
A locally univalent function $f$ is sense-preserving when $J_f (z) > 0$ in $\Omega$. Thus, a harmonic mapping $f$ is locally univalent and sense-preserving
in $\Omega$ if, and only if, $J_f (z) > 0$ in $\Omega$.
This is equivalent to the conditions that $h'\not=0$ in $\Omega$ and the dilatation $\omega_f:=g'/h'$ of $f$ has the property that $|\omega_f| < 1$ in $\Omega$ (see \cite{8a,20P,18}).\\[2mm]
\indent A locally univalent and sense-preserving harmonic mapping $f = h + \ol{g}$ on a domain $\Omega$ is said to be a $K$-quasiconformal harmonic mapping if it satisfies the condition $|\omega_f (z)|\le k < 1$ for all $z\in\Omega$, where $K = (1+k)/(1-k) \geq 1$ (see \cite{12,22}). 
It is evident that the limit of $k$ approaching $1$ correlates with the scenario where $K$ is approaching $+\infty$. 
The results on harmonic extensions of the classical Bohr theorem have been studied in \cite{9,15,17,20,21}.\\[2mm]
For a harmonic mapping in $\Omega_{\gamma}$, Evdoridis {\it et al.} \cite{9a} studied the following and obtained the Bohr inequality for its restriction to $\mathbb{D}$.
\begin{theoF}\cite{9a} Let $f = h + \ol g$ be a harmonic mapping in $\Omega_{\gamma}$, with $|h(z)| \leq 1$ on $\Omega_{\gamma}$. If $h(z) = \sum_{n=0}^{\infty} \alpha_nz^n$ and $g(z) = \sum_{n=1}^{\infty} \beta_nz^n$ for $z\in\mathbb{D}$ and $|g'(z)| \leq k|h'(z)|$ for some $k \in [0, 1]$, then 
\beas \sum_{n=0}^\infty|\alpha_n|\rho^n+\sum_{n=1}^\infty|\beta_n|\rho^n\leq 1\quad \text{for}\quad \rho\leq \rho_0:=\frac{1+\gamma}{3+2k+\gamma}.\eeas
The radius $\rho_0$ is the best possible.\end{theoF}
\begin{corA}\cite{9a} Let $f = h + \ol g$ be a harmonic mapping in $\Omega_{\gamma}$, with $|h(z)| \leq 1$ on $\Omega_{\gamma}$. If $h(z) = \sum_{n=0}^{\infty} \alpha_nz^n$ and $g(z) = \sum_{n=1}^{\infty} \beta_nz^n$ for $z\in\mathbb{D}$ and $f=h+\ol g$ is sense-preserving in $\mathbb{D}$, then
\beas \sum_{n=0}^\infty|\alpha_n|\rho^n+\sum_{n=1}^\infty|\beta_n|\rho^n\leq 1\quad \text{for}\quad\rho\leq \rho_0:=\frac{1+\gamma}{5+\gamma}.\eeas
The radius $\rho_0$ is the best possible.
\end{corA}
 In this context, Ahamed {\it et al.} \cite{1a} have obtained the following improved versions of \textrm{Theorem B} for the class $\mathcal{B}(\Omega_{\gamma})$ with its restriction to $\mathbb{D}$.
\begin{theoG}\cite{1a} For $\gamma\in[0,1)$, and $m\in\mathbb{N}\setminus\{1\}$, let $f \in\mathcal{B}(\Omega_{\gamma})$ with $f(z)=\sum_{n=0}^{\infty} \alpha_nz^n$ for $z\in\mathbb{D}$. Then,
\beas |\alpha_0|+\sum_{n=1}^\infty \left(|\alpha_n|+\tau\frac{|\alpha_n|^m}{(1-\gamma)^{(m-1)n}}\right)\rho^n\leq 1\;\;\text{ for}\;\; \rho \leq \rho_0:= \frac{1 + \gamma}{3 + \gamma},\eeas
where \beas \tau=\frac{(1-\gamma)^m(3+\gamma)-(1-\gamma^2)}{8(m-1)}\;\;\text{for}\;\;0\leq \gamma<1.\eeas
 Furthermore, the quantities $\tau$ and $(1+\gamma)/(3+\gamma)$ cannot be improved.\end{theoG}
\begin{theoH} \cite{1a} For $\gamma\in[0,1)$, and $m\in\mathbb{N}\setminus\{1\}$, let $f \in\mathcal{B}(\Omega_{\gamma})$ with $f(z)=\sum_{n=0}^{\infty} \alpha_nz^n$ for $z\in\mathbb{D}$. Then, 
\beas \sum\limits_{n=0}^{\infty}|\alpha_n|\rho^n+\left(\frac{8}{9}-\frac{27}{64}\lambda\right)\left(\frac{S_{\rho(1-\gamma)}}{\pi}\right)+\lambda\left(\frac{S_{\rho(1-\gamma)}}{\pi}\right)^2\leq 1 \;\;\text{for}\;\;\rho\leq \rho_0:=(1+\gamma)/(3+\gamma).\eeas
Furthermore, the radius $\rho_0$ is sharp, and the bounds of $\lambda$ and $8/9-27\lambda/64$ cannot be improved.\end{theoH}
In 2022, Ahamed {\it et al.} \cite{1a} established the following refined Bohr inequality for the class $\mathcal{B}(\Omega_{\gamma})$ with its restriction to $\mathbb{D}$.
\begin{theoJ} For $\gamma\in[0,1)$ and $N\in\mathbb{N}$, let $f \in\mathcal{B}(\Omega_{\gamma})$ with $f(z)=\sum_{n=0}^{\infty} \alpha_nz^n$ for $z\in\mathbb{D}$. Then, 
\beas \sum_{n=1}^\infty |\alpha_n|\rho^n+\left(\frac{1}{1+|\alpha_1|}+\frac{\rho}{1-\rho}\right) \sum_{n=2}^{\infty} |\alpha_n|^2\rho^{2(n-1)}\leq 1\;\;\text{ for}\;\; \rho \leq \rho_0:= \rho_{\gamma} := \frac{1 + \gamma}{3 + \gamma}. \eeas
The constant $\rho_0$ cannot be improved.
\end{theoJ}
\section{Main results}
\noindent In the shifted disk $\Omega_{\gamma}$, Evdoridis {\it et al.} {\cite{9a} have recently studied sharply improved versions of the Bohr inequality for analytic and harmonic mappings, 
with its restriction to the unit disk $\mathbb{D}$. Subsequently,  Ahamed {\it et al.} \cite{1a} have established a number of sharply improved and refined versions of Bohr's inequality for analytic functions in $\Omega_{\gamma}$, with the restriction to the unit disk $\Bbb{D}$.\\[2mm]
\indent In this paper, we define the Bohr radius for the class $\mathcal{B}(\Omega_{\gamma})$, motivated by Fournier and Ruscheweyh \cite{10}. It is defined as the number 
$B_{\Omega_{\gamma}}\in (0, 1)$ such that 
 \beas B_{\Omega_{\gamma}}=\sup\left\{\rho\in (0,1):M_f(\rho)\leq 1\;\text{for} \;f(z)=\sum_{n=0}^{\infty}\alpha_n\left(z+\frac{\gamma}{1-\gamma}\right)^n\in\mathcal{B}(\Omega_{\gamma}),z\in\Omega_{\gamma}\right\},\eeas 
where $M_f(\rho)=\sum_{n=0}^{\infty}|\alpha_n|\left(\rho/(1-\gamma)\right)^n$ with $|\gamma+(1-\gamma)z|=\rho$, is the majorant series associated with the analytic functions $f\in\mathcal{B}(\Omega_{\gamma})$. It is well known that $B_{\mathbb{D}} = 1/3$ when $\Omega_{\gamma}=\mathbb{D}$.\\[2mm]
\indent 
For the class of analytic and harmonic mapping in $\Omega_{\gamma}$, we obtain sharply Bohr radius, refined Bohr radius, and improved Bohr radius in this paper, without its restriction to the unit disk $\mathbb{D}$.\\[2mm]
The following are key lemmas of this paper and will be used to prove the main results.
\begin{lem}\cite{25}\label{le1} For $f\in\mathcal{B}(\mathbb{D})$, then we have 
\beas \frac{\left|f^{(n)}(\alpha)\right|}{n!}\leq \frac{1-|f(\alpha)|^2}{(1-|\alpha|)^{n-1}(1-|\alpha|^2)}\;\;\text{for each}\;\; n\geq 1\;\;\text{and}\;\; \alpha\in\mathbb{D}.\eeas\end{lem}
\begin{lem}\cite{9a}\label{le3} Suppose that $h(z)=\sum_{n=0}^\infty a_nz^n$ and $g(z)=\sum_{n=0}^\infty b_nz^n$ are two analytic functions in $\mathbb{D}$ such that $|g’(z)|\leq k|h’(z)|$ in $\mathbb{D}$ and for some $k\in [0,1)$ with $|h(z)|\leq 1$. Then,
\beas\sum_{n=1}^\infty n|b_n|^2\rho^{2n}\leq k^2\sum_{n=1}^\infty n|a_n|^2\rho^{2n}\;\;\text{and}\;\;\sum_{n=1}^\infty |b_n|^2\rho^n\leq k^2 \sum_{n=1}^\infty |a_n|^2\rho^n\;\;\text{for}\;\;|z|=\rho<1.\eeas
\end{lem}
\begin{lem}\label{le4} Suppose that $h(z)=\sum_{n=0}^\infty a_n\left(z+\frac{\gamma}{1-\gamma}\right)^n$ and $g(z)=\sum_{n=0}^\infty b_n\left(z+\frac{\gamma}{1-\gamma}\right)^n$ are two analytic functions in $\Omega_{\gamma}$ such that $|g’(z)|\leq k|h’(z)|$ in $\Omega_{\gamma}$ and for some $k\in [0,1)$ with $|h(z)|\leq 1$. Then, for $\left|(1-\gamma)z+\gamma\right|=\rho<1$, we have
\beas&&\sum_{n=1}^\infty n\frac{|b_n|^2}{(1-\gamma)^{2n}}\rho^{2n}\leq k^2 \sum_{n=1}^\infty n\frac{|a_n|^2}{(1-\gamma)^{2n}}\rho^{2n}\\[2mm]\text{and}
&&\sum_{n=1}^\infty \frac{|b_n|^2}{(1-\gamma)^{2n}}\rho^n\leq k^2 \sum_{n=1}^\infty \frac{|a_n|^2}{(1-\gamma)^{2n}}\rho^n.\eeas
\end{lem}
\begin{proof} Let $\Phi : \mathbb{D}\to\Omega_{\gamma}$ be a function defined by $\Phi(z)=(z-\gamma)/(1-\gamma)$. Since $h,g: \Omega_{\gamma}\to\ol{\mathbb{D}}$, so
the composition $\varphi_1=h\circ \Phi$ and $\varphi_2=g\circ \Phi$ are analytic in $\mathbb{D}$. 
Since $z\in\Omega_{\gamma}$, so we write $z=(\xi-\gamma)/(1-\gamma)$ for $\xi\in\mathbb{D}$. Thus, we have 
\bs\beas\varphi_1(\xi)=h\left(\frac{\xi-\gamma}{1-\gamma}\right)=\sum\limits_{n=0}^\infty \frac{a_n}{(1-\gamma)^n}\xi^n\;\;\text{and}\;\;\varphi_2(\xi)=g\left(\frac{\xi-\gamma}{1-\gamma}\right)=\sum\limits_{n=0}^\infty \frac{b_n}{(1-\gamma)^n}\xi^n\quad\text{for}\;\;\xi\in\mathbb{D}.\eeas\es
Since \beas\left|g’\left(\frac{\xi-\gamma}{1-\gamma}\right)\right|\leq k \left|h’\left(\frac{\xi-\gamma}{1-\gamma}\right)\right|\;\text{for}\; \xi\in\mathbb{D},\eeas  
we have 
$|\varphi_2’(\xi)|\leq k|\varphi_1’(\xi)|$ for $\xi\in\mathbb{D}$. In light of \textrm{Lemma \ref{le3}}, we have
\beas\sum_{n=1}^\infty n\frac{|b_n|^2}{(1-\gamma)^{2n}}\rho^{2n}\leq k^2 \sum_{n=1}^\infty n\frac{|a_n|^2}{(1-\gamma)^{2n}}\rho^{2n}\;\text{and}\;
\sum_{n=1}^\infty \frac{|b_n|^2}{(1-\gamma)^{2n}}\rho^n\leq k^2 \sum_{n=1}^\infty \frac{|a_n|^2}{(1-\gamma)^{2n}}\rho^n\eeas
for $|\xi|=\rho<1$, {\it i.e.}, $|\gamma+(1-\gamma)z|=\rho<1$.
\end{proof}
\begin{lem}\label{le2}Let $f$ be analytic in $\Omega_{\gamma}$, bounded by $1$ with the series expansion $f(z)=\sum_{n=0}^\infty a_n\left(z+\frac{\gamma}{1-\gamma}\right)^n$ in $\Omega_{\gamma}$. Then, $|a_n|\leq (1-\gamma)^n(1-|a_0|^2)$ for $n\geq 1$.\end{lem}
\begin{proof} 
Let $\Phi : \mathbb{D}\to\Omega_{\gamma}$ be a function defined by $\Phi(z)=(z-\gamma)/(1-\gamma)$. Since $f : \Omega_{\gamma}\to\ol{\mathbb{D}}$ is analytic, so
the composition $g=f\circ \Phi$ is analytic in $\mathbb{D}$. Thus, 
\beas g(z)=f\left(\frac{z-\gamma}{1-\gamma}\right)=\sum\limits_{n=0}^\infty \frac{a_n}{(1-\gamma)^n}z^n\eeas
with $g(0)=a_0$ and 
\beas a_n=\frac{g^{(n)}(0)}{n!}(1-\gamma)^n.\eeas
In view of \textrm{Lemma \ref{le1}}, we have 
\beas |a_n|\leq (1-\gamma)^n\left(1-|g(0)|^2\right)=(1-\gamma)^n\left(1-|a_0|^2\right).\eeas
\end{proof}
For the class of analytic functions $\mathcal{B}(\Omega_{\gamma})$, we obtain a sharp refined version of Bohr inequality as follows:
\begin{theo}\label{th0} Let $f$ be an analytic function in $\Omega_{\gamma}$ with $|f(z)|\leq 1$ on $\Omega_{\gamma}$. If $f(z)=\sum_{n=0}^\infty a_n\left(z+\frac{\gamma}{1-\gamma}\right)^n$ in $\Omega_{\gamma}$, then
\bea\label{rreq}\sum_{n=0}^\infty \frac{|a_n|}{(1-\gamma)^n}\rho^n+\left(\frac{1}{1+|a_0|}+\frac{\rho}{1-\rho}\right)\sum_{n=1}^\infty\frac{|a_n|^2}{(1-\gamma)^{2n}}\rho^{2n}\leq 1\eea 
for $\left|(1-\gamma)z+\gamma\right|=\rho\leq \rho_0=1/3$. The number $\rho_0$ is the best possible. \end{theo}
Note that in \textrm{Theorem \ref{th0}} we obtain $\rho_0=1/3$, which is same as the classical Bohr radius $1/3$.\\[2mm]
\indent In the following, we obtain a sharp Bohr inequality for harmonic mapping on the shifted disk $\Omega_{\gamma}$.
\begin{theo}\label{th1} Let $f=h+\ol{g}$ be harmonic mapping in $\Omega_{\gamma}$ with $|h(z)|\leq 1$ on $\Omega_{\gamma}$. If $h(z)=\sum_{n=0}^\infty a_n\left(z+\frac{\gamma}{1-\gamma}\right)^n$ and $g(z)=\sum_{n=1}^\infty b_n\left(z+\frac{\gamma}{1-\gamma}\right)^n$ in $\Omega_{\gamma}$ and $|g’(z)|\leq k|h’(z)|$ for some $k\in[0,1)$, then
\beas\sum_{n=0}^\infty \frac{|a_n|}{(1-\gamma)^n}\rho^n+\sum_{n=1}^\infty\frac{|b_n|}{(1-\gamma)^n}\rho^n\leq 1\;\text{for}\;\left|(1-\gamma)z+\gamma\right|=\rho\leq \rho_0=1/(2k+3).\eeas
The number $\rho_0$ is the best possible.
\end{theo}
Letting $k \to 1$ in \textrm{Theorem \ref{th1}}, we obtain the following sharp harmonic analogues of the classical Bohr inequality on the shifted disk $\Omega_{\gamma}$.
\begin{cor} Let $f=h+\ol{g}$ be harmonic mapping in $\Omega_{\gamma}$ with $|h(z)|\leq 1$ on $\Omega_{\gamma}$. If $h(z)=\sum_{n=0}^\infty a_n\left(z+\frac{\gamma}{1-\gamma}\right)^n$, $g(z)=\sum_{n=1}^\infty b_n\left(z+\frac{\gamma}{1-\gamma}\right)^n$ in $\Omega_{\gamma}$ and $f(z)$ is sense-preserving in $\Omega_{\gamma}$, then
\beas\sum_{n=0}^\infty \frac{|a_n|}{(1-\gamma)^n}\rho^n+\sum_{n=1}^\infty\frac{|b_n|}{(1-\gamma)^n}\rho^n\leq 1\;\text{for}\;\left|(1-\gamma)z+\gamma\right|=\rho\leq \rho_0=1/5.\eeas
The number $\rho_0$ is the best possible.
\end{cor}
In the following results, we obtain several sharp versions of Bohr inequalities in improved form for harmonic mapping on the shifted disk $\Omega_{\gamma}$.
\begin{theo}\label{th2} Let $f=h+\ol{g}$ be harmonic mapping in $\Omega_{\gamma}$ with $|h(z)|\leq 1$ on $\Omega_{\gamma}$. If $h(z)=\sum_{n=0}^\infty a_n\left(z+\frac{\gamma}{1-\gamma}\right)^n$, $g(z)=\sum_{n=1}^\infty b_n\left(z+\frac{\gamma}{1-\gamma}\right)^n$ in $\Omega_{\gamma}$ and $|g’(z)|\leq k|h’(z)|$ for some $k\in[0,1)$, then
\beas\sum_{n=0}^\infty \frac{|a_n|}{(1-\gamma)^n}\rho^n+\sum_{n=1}^\infty\frac{|b_n|}{(1-\gamma)^n}\rho^n+ \frac{2(k+2)^2(k+1)^2}{(2k+3)^2}\left(\frac{S_\rho^\gamma(h)}{\pi}\right)\leq 1\eeas 
for $\left|(1-\gamma)z+\gamma\right|=\rho\leq \rho_0=1/(2k+3)$, where $S_\rho^\gamma(h)$ is the area of the image $\mathbb{D}(\gamma/(\gamma-1);\rho/(1-\gamma))$ under the mapping $h$.
The numbers $2(k+2)^2(k+1)^2/(2k+3)^2$ and $\rho_0$ cannot be replaced by a larger value.
\end{theo}
Letting $k \to 1$ in \textrm{Theorem \ref{th2}}, we obtain the following sharp version of Bohr inequality in improved form for harmonic mapping on the shifted disk $\Omega_{\gamma}$.
\begin{cor} Let $f=h+\ol{g}$ be harmonic mapping in $\Omega_{\gamma}$ with $|h(z)|\leq 1$ on $\Omega_{\gamma}$. If $h(z)=\sum_{n=0}^\infty a_n\left(z+\frac{\gamma}{1-\gamma}\right)^n$, $g(z)=\sum_{n=1}^\infty b_n\left(z+\frac{\gamma}{1-\gamma}\right)^n$ in $\Omega_{\gamma}$ and $f(z)$ is sense-preserving in $\Omega_{\gamma}$, then
\beas\sum_{n=0}^\infty \frac{|a_n|}{(1-\gamma)^n}\rho^n+\sum_{n=1}^\infty\frac{|b_n|}{(1-\gamma)^n}\rho^n+ \frac{72}{25}\frac{S_\rho^\gamma(h)}{\pi}\leq 1\eeas 
for $\left|(1-\gamma)z+\gamma\right|=\rho\leq \rho_0=1/5$, where $S_\rho^\gamma(h)$ is the area of the image $\mathbb{D}(\gamma/(\gamma-1);\rho/(1-\gamma))$ under the mapping $h$.
The numbers $72/25$ and $\rho_0$ cannot be replaced by a larger value.
\end{cor}
\begin{theo}\label{th3} Let $f = h + \ol g$ be a harmonic mapping in $\Omega_{\gamma}$, with $\left|h(z)\right| \leq 1$ on $\Omega_{\gamma}$. If $h(z)=\sum_{n=0}^\infty a_n\left(z+\frac{\gamma}{1-\gamma}\right)^n$, $g(z)=\sum_{n=1}^\infty b_n\left(z+\frac{\gamma}{1-\gamma}\right)^n$ in $\Omega_{\gamma}$ and $|g'(z)| \leq k|h'(z)|$ for some $k \in [0, 1)$, then 
\beas\sum_{n=0}^\infty\frac{|a_n|}{(1-\gamma)^n}\rho^n+\sum_{n=1}^\infty\frac{|b_n|}{(1-\gamma)^n}\rho^n+\frac{2(k+2)^2(k+1)}{(1-k)(2k+3)^2}\left(\frac{S_\rho^\gamma(f)}{\pi}\right)\leq 1, \eeas
for $\left|(1-\gamma)z+\gamma\right|=\rho \leq \rho_0=1/(2k+3)$, where $S_\rho^\gamma(f)$ is the area of the image of the disk $\mathbb{D}(\gamma/(\gamma-1);\rho/(1-\gamma))$ under the mapping $f$.
The numbers $2(k+2)^2(k+1)/((1-k)(2k+3)^2)$ and $\rho_0$ cannot be replaced by a larger value.
\end{theo}
\section{Proofs of the main results}
\begin{proof}[\bf{Proof of Theorem \ref{th0}}]  
Let $f(z)$ be analytic on $\Omega_{\gamma}$ with $|f(z)|\leq 1$. In view of \textrm{Lemma \ref{le2}}, we have $|a_n|\leq (1-\gamma)^n (1-|a_0|^2)$ for $n\geq 1$. Then
\beas &&\sum_{n=0}^\infty \frac{|a_n|}{(1-\gamma)^n}\rho^n+\left(\frac{1}{1+|a_0|}+\frac{\rho}{1-\rho}\right)\sum_{n=1}^\infty\frac{|a_n|^2}{(1-\gamma)^{2n}}\rho^{2n}\\
&&\leq |a_0|+(1-|a_0|^2)\sum_{n=1}^\infty \rho^n+ \left(\frac{1}{1+|a_0|}+\frac{\rho}{1-\rho}\right)(1-|a_0|^2)^2\sum_{n=1}^\infty \rho^{2n}\\
&&=|a_0|+\frac{(1-|a_0|^2)\rho}{1-\rho}+\left(\frac{1}{1+|a_0|}+\frac{\rho}{1-\rho}\right)\frac{(1-|a_0|^2)^2\rho^2}{1-\rho^2}.\eeas
Let $|a_0|=a\in[0,1]$. Then
\beas \sum_{n=0}^\infty \frac{|a_n|}{(1-\gamma)^n}\rho^n+\left(\frac{1}{1+|a_0|}+\frac{\rho}{1-\rho}\right)\sum_{n=1}^\infty\frac{|a_n|^2}{(1-\gamma)^{2n}}\rho^{2n}\leq 1+\xi(a),\eeas
where 
\beas&& \xi(a)=\frac{(1-a^2)\rho}{1-\rho}+\left(\frac{1}{1+a}+\frac{\rho}{1-\rho}\right)\frac{(1-a^2)^2\rho^2}{1-\rho^2}-(1-a)\\
&&=A(1-a^2)+B(1-a^2)(1-a)+C(1-a^2)^2-(1-a),\eeas
where $A=\rho/(1-\rho)\geq 0$, $B=\rho^2/(1-\rho^2)\geq 0$, $C=\rho^3/((1-\rho)(1-\rho^2))\geq 0$. Now
\beas&&\xi(0)=A+B+C-1\;\text{and}\;\xi(1)=0\\
 &&\xi'(a)=1-2aA+B(3a^2-2a-1)-4C(a-a^3)\\
&&\xi''(a)=-2A+B(6a-2)-4C(1-3a^2)\;\text{and}\;\xi'''(a)=6B+24aC\geq 0. \eeas
Therefore, $\xi''(a)$ is a monotonically increasing function of $a$ in $[0,1]$ and it follows that 
\beas\xi''(a)\leq \xi''(1)&=&-2A+4B+8C\\
&=&-\frac{2\rho}{1-\rho}+\frac{4\rho^2}{1-\rho^2}+\frac{8\rho^3}{(1-\rho)(1-\rho^2)}\\
&=&\frac{2\rho}{(1-\rho)(1-\rho^2)}\left(-1+\rho^2+2\rho(1-\rho)+4\rho^2\right)\\
&=&\frac{2\rho}{(1-\rho)(1-\rho^2)}(\rho+1)(3\rho-1)\leq 0\eeas
for $\rho\leq \rho_0=1/3$.
Therefore, $\xi''(a)\leq 0$ for $\rho\leq \rho_0$. Hence, $\xi'(a)$ is a monotonically decreasing function in $[0,1]$ and it follows that 
\beas \xi'(a)\geq \xi'(1)=1-2A=1-\frac{2\rho}{1-\rho}=\frac{1}{1-\rho}\left(1-3\rho\right)\geq 0\;\text{for}\;\rho\leq \rho_0=1/3.\eeas
Therefore, $\xi(a)$ is a monotonically increasing function in $[0,1]$ and hence $\xi(a)\leq \xi(1)=0$ for $a\in[0,1]$ and $\rho\leq \rho_0=1/3$.\\[2mm]
To prove the sharpness of the result, we consider the function $f_1(z)$ in $\Omega_\gamma$ such that $f_1=\psi\circ\Phi_1$, where $\Phi_1 : \Omega_{\gamma}\to\mathbb{D}$ defined by $\Phi_1(z)=\gamma+(1-\gamma)z$ and $\psi : \mathbb{D}\to\mathbb{D}$ defined by $\psi(z)=(a-z)/(1-az)$ for $a\in(0,1)$ and $\gamma\in[0,1)$. Therefore,
\beas f_1(z)=\frac{a-(1-\gamma)\left(z+\frac{\gamma}{1-\gamma}\right)}{1-a(1-\gamma)\left(z+\frac{\gamma}{1-\gamma}\right)}
=A_0-\sum_{n=1}^\infty A_n \left(z+\frac{\gamma}{1-\gamma}\right)^n\;\text{for}\;z\in\Omega_\gamma,\eeas	
where $A_0=a$ and $A_n=a^{n-1}(1-a^2)(1-\gamma)^n$.
Thus,
\beas &&S:=\sum_{n=0}^\infty \frac{|A_n|}{(1-\gamma)^n}\rho^n+\left(\frac{1}{1+|A_0|}+\frac{\rho}{1-\rho}\right)\sum_{n=1}^\infty\frac{|A_n|^2}{(1-\gamma)^{2n}}\rho^{2n}\\
&&=a+\frac{(1-a^2)}{a}\sum_{n=1}^\infty \left(a\rho\right)^n+\left(\frac{1}{1+a}+\frac{\rho}{1-\rho}\right)\frac{(1-a^2)^2}{a^2}\sum_{n=1}^\infty(a\rho)^{2n}\\
&&=a+\frac{(1-a^2)\rho}{1-a\rho}+\left(\frac{1}{1+a}+\frac{\rho}{1-\rho}\right)\frac{(1-a^2)^2\rho^2}{1-a^2\rho^2}\\
&&=1+\frac{(1-a^2)\rho}{1-a\rho}+\left(\frac{1}{1+a}+\frac{\rho}{1-\rho}\right)\frac{(1-a^2)^2\rho^2}{1-a^2\rho^2}-(1-a)\\
&&=1-\frac{(1-a)}{1-a\rho}\left(1-a\rho-(1+a)\rho-\left(\frac{1}{1+a}+\frac{\rho}{1-\rho}\right)\frac{(1-a^2)(1+a)\rho^2}{1+a\rho}\right)\\
&&=1-\frac{(1-a)}{1-a\rho}F_1(a, \rho),\eeas
where 
\beas F_1(a, \rho)=1-a\rho-(1+a)\rho-\left(\frac{1}{1+a}+\frac{\rho}{1-\rho}\right)\frac{(1-a^2)(1+a)\rho^2}{1+a\rho}.\eeas
By differentiating partially $F_1(a, \rho)$ with respect to $\rho$, we have
\beas \frac{\pa}{\pa \rho}F_1(a,\rho)=-(1+2a)-\left(\frac{1}{1+a}+\frac{\rho}{1-\rho}\right)\frac{(1-a^2)(1+a)(1+a\rho^2)}{(1+a\rho)^2}-\frac{(1-a^2)(1+a)\rho^2}{(1+a\rho)(1-\rho)^2}<0.\eeas
Thus, $F_1(a, \rho)$ is strictly decreasing function of $\rho\in(0,1)$. Therefore, for $\rho>\rho_0=1/3$, we have
\beas F_1(a,\rho)<F_1(a,\rho_0)=1-\frac{a}{3}-\frac{1+a}{3}-\left(\frac{1}{1+a}+\frac{1}{2}\right)\frac{(1-a^2)(1+a)}{3(3+a)}\to 0\;\text{as}\;a\to1.\eeas 
 Hence $S:=1-(1-a)F_1(a,\rho)/(1-a\rho)>1$ for $\rho>\rho_0$. 												
This shows that $\rho_0$ is best possible. This completes the proof.
\end{proof}
\begin{proof}[\bf{Proof of Theorem \ref{th1}}]  
Let $h(z)$ be analytic in $\Omega_{\gamma}$ with $|h(z)|\leq 1$. In view of \textrm{Lemma \ref{le2}}, we have $|a_n|\leq (1-\gamma)^n (1-|a_0|^2)$ for $n\geq 1$. Thus,
\bea\label{a1}\sum_{n=0}^\infty \frac{|a_n|}{(1-\gamma)^n}\rho^n\leq |a_0|+(1-|a_0|^2)\sum_{n=1}^\infty \rho^n=|a_0|+\frac{(1-|a_0|^2)\rho}{1-\rho}.\eea
In view of the fact that $|g’(z)|\leq k|h’(z)|$ on $\Omega_{\gamma}$, it follows from \textrm{Lemma \ref{le4}} that 
\beas  \sum_{n=1}^\infty \frac{|b_n|^2}{(1-\gamma)^{2n}}\rho^n&\leq& k^2 \sum_{n=1}^\infty \frac{|a_n|^2}{(1-\gamma)^{2n}}\rho^n
\leq \frac{k^2(1-|a_0|^2)^2\rho}{1-\rho}\eeas
for $\left|(1-\gamma)z+\gamma\right|=\rho<1$. Using Cauchy-Schwarz inequality, we have
\bea\label{a2}\sum_{n=1}^\infty \frac{|b_n|}{(1-\gamma)^n}\rho^n&\leq& \left(\sum_{n=1}^\infty \frac{|b_n|^2}{(1-\gamma)^{2n}}\rho^n\right)^{1/2}\left(\sum_{n=1}^\infty \rho^n\right)^{1/2}\nonumber\\
&\leq&\left(\frac{k^2(1-|a_0|^2)^2\rho}{1-\rho}\right)^{1/2}\left(\frac{\rho}{1-\rho}\right)^{1/2}= \frac{k(1-|a_0|^2)\rho}{(1-\rho)}.\eea
Let $|a_0|=a\in[0,1]$.
From (\ref{a1}) and (\ref{a2}), we have
\bea\label{a4}\sum_{n=0}^\infty \frac{|a_n|}{(1-\gamma)^n}\rho^n+\sum_{n=1}^\infty\frac{|b_n|}{(1-\gamma)^n}\rho^n\leq a+(1+k)\frac{(1-a^2)\rho}{1-\rho}=1+\xi(a),\eea
where 
\beas&& \xi(a)=(1+k)\frac{(1-a^2)\rho}{1-\rho}-(1-a)=\frac{1-a}{1-\rho}\left[(1+k)(1+a)\rho-(1-\rho)\right]\\
&&=\frac{1-a}{1-\rho}\left(((1+k)(1+a)+1)\rho-1\right).\eeas
It is easy to see that
\beas \xi(0)=\frac{(1+k)\rho}{1-\rho}-1, \xi(1)=0, \xi'(a)=1-\frac{2(1+k)a\rho}{1-\rho}\;\text{and}\;\xi''(a)=-\frac{2(1+k)\rho}{1-\rho}\leq 0.\eeas
Therefore, $\xi'(a)$ is a monotonically decreasing function of $a$ in $[0,1]$ and it follows that 
\beas\xi'(a)\geq \xi'(1)=1-\frac{2(1+k)\rho}{1-\rho}\geq 0\quad\text{for}\quad\rho\leq \rho_0=1/(2k+3).\eeas
Hence, $\xi(a)$ is a monotonically increasing function of $a$ for $\rho\leq \rho_0$ and $\gamma\in[0,1)$, thus $\xi(a)\leq \xi(1)=0$ for $\rho\leq \rho_0$.\\[2mm]
\indent To prove the sharpness of the result, we consider the function $f_2(z)=h_2(z)+\ol{g_2(z)}$ in $\Omega_\gamma$, where $h_2=\psi\circ\Phi_2$ with $\Phi_2 : \Omega_{\gamma}\to\mathbb{D}$ defined by $\Phi_2(z)=\gamma+(1-\gamma)z$ and $\psi : \mathbb{D}\to\mathbb{D}$ defined by $\psi(z)=(a-z)/(1-az)$ for $a\in(0,1)$ and $\gamma\in[0,1)$. Therefore,
\beas h_2(z)
=\frac{a-(1-\gamma)\left(z+\frac{\gamma}{1-\gamma}\right)}{1-a(1-\gamma)\left(z+\frac{\gamma}{1-\gamma}\right)}
=A_0-\sum_{n=1}^\infty A_n \left(z+\frac{\gamma}{1-\gamma}\right)^n\quad\text{for}\quad z\in\Omega_\gamma,\eeas	
where $A_0=a$, $A_n=a^{n-1}(1-a^2)(1-\gamma)^n$  and $g_2(z)=-k\sum_{n=1}^\infty A_n \left(z+\frac{\gamma}{1-\gamma}\right)^n$.
Thus 
\beas S:=\sum_{n=0}^\infty \frac{|a_n|}{(1-\gamma)^n}\rho^n+\sum_{n=1}^\infty\frac{|b_n|}{(1-\gamma)^n}\rho^n=A_0+\frac{(1-a^2)(1+k)}{a}\sum_{n=1}^\infty \left(a\rho\right)^n
=1-\frac{F_2(a,\rho)}{(1-a\rho)},\eeas
where 
\beas&& F_2(a,\rho)=(1-a)(1-a\rho)-(1-a^2)(1+k)\rho=(1-a)\left((1-a\rho)-(1+a)(1+k)\rho\right).\eeas
Let 
\bea\label{wre4} W(a,\rho)=(1-a\rho)-(1+a)(1+k)\rho.\eea 
It is easy to see that $S>1$ if $F_2(a,\rho)<0$ and $F_2(a,\rho)<0$ if $W(a,\rho)<0$. Allowing $a\to 1$ in (\ref{wre4}), we have
\beas W(1,\rho)=(1-\rho)-2(1+k)\rho=1-(2k+3)\rho<0\;\text{for}\;\rho>\rho_0=1/(2k+3).\eeas
This shows that $\rho_0$ is best possible.
 This completes the proof.
\end{proof}
\begin{proof}[\bf{Proof of Theorem \ref{th2}}]
Let $h(z)$ be analytic on $\Omega_{\gamma}$ with $|h(z)|\leq 1$. In view of \textrm{Lemma \ref{le2}}, we have $|a_n|\leq (1-\gamma)^n (1-|a_0|^2)$ for $n\geq 1$. Thus,
\bea\label{aa1}\sum_{n=0}^\infty \frac{|a_n|}{(1-\gamma)^n}\rho^n\leq |a_0|+(1-|a_0|^2)\sum_{n=1}^\infty \rho^n=|a_0|+\frac{(1-|a_0|^2)\rho}{1-\rho}.\eea
Using similar argument as in the proof of \textrm{Theorem \ref{th1}} and by using \textrm{Lemma \ref{le4}} with $|g'(z)|\leq k|h'(z)|$, we have
\bea\label{aa2} \sum_{n=1}^\infty \frac{|b_n|}{(1-\gamma)^n}\rho^n\leq \frac{k(1-|a_0|^2)\rho}{(1-\rho)}\;\;\text{for}\;\left|z(1-\gamma)+\gamma\right|=\rho.\eea
We consider the function $\Phi :\mathbb{D}\to \Omega_{\gamma}$ defined by $\Phi(z)=(z-\gamma)/(1-\gamma)$. Since $h: \Omega_{\gamma}\to\ol{\mathbb{D}}$, so
the composition $\varphi_1=h\circ \Phi$ is analytic in $\mathbb{D}$. 
Since $z\in\Omega_{\gamma}$, so we write $z=(\xi-\gamma)/(1-\gamma)$ for $\xi\in\mathbb{D}$. 
Thus, we have 
\beas\varphi(\xi)=h\left(\frac{\xi-\gamma}{1-\gamma}\right)=\sum\limits_{n=0}^\infty \frac{a_n}{(1-\gamma)^n}\xi^n=\sum\limits_{n=0}^\infty d_n\xi^n\;\text{for}\;\xi\in\mathbb{D},\eeas
where $d_n=a_n/(1-\gamma)^n$ for $n\geq 0$. It is known that for arbitrary entire function $G(z)=\sum_{n=0}^\infty \alpha_n z^n$ for $z\in\mathbb{D}$, the area functional is given by
\beas \frac{S_\rho}{\pi}=\frac{1}{\pi}\text{Area}\;\left(G(\mathbb{D}(0;\rho))\right)=\frac{1}{\pi}\iint_{|z|<\rho}|G'(z)|dx dy=\sum_{n=1}^\infty n|\alpha_n|^2\rho^{2n}.\eeas
Therefore, in view of \textrm{Lemma \ref{le2}}, we have  
\beas \frac{1}{\pi}\text{Area}\;\left(\varphi(\mathbb{D}(0;\rho))\right)=\sum_{n=1}^\infty n|d_n|^2\rho^{2n}=\sum_{n=1}^\infty \frac{n|a_n|^2}{(1-\gamma)^{2n}}\rho^{2n}\leq (1-|a_0|^2)^2\frac{\rho^2}{(1-\rho^2)^2}.\eeas
It is easy to see that 
\beas \text{Area}\left(\varphi(\mathbb{D}(0;\rho))\right)=\text{Area}\left(h\left(\Phi(\mathbb{D}(0;\rho)\right))\right)=\text{Area}\left[h\left(\mathbb{D}(\gamma/(\gamma-1);\rho/(1-\gamma))\right)\right]= S_\rho^\gamma(h).\eeas
Let $|a_0|=a\in[0,1]$. From (\ref{aa1}), (\ref{aa2}) and (\ref{a4}), we have 
\bea\label{a4}&&\sum_{n=0}^\infty \frac{|a_n|}{(1-\gamma)^n}\rho^n+\sum_{n=1}^\infty\frac{|b_n|}{(1-\gamma)^n}\rho^n+K\frac{S_\rho^\gamma(h)}{\pi}\nonumber\\
&&\leq a+(1+k)\frac{(1-a^2)\rho}{1-\rho}+K(1-a^2)^2\frac{\rho^2}{(1-\rho^2)^2}=1+\xi(\rho),\eea
where
\beas \xi(\rho)&=&\frac{(1+k)(1-a^2)\rho}{1-\rho}+K(1-a^2)^2\frac{\rho^2}{(1-\rho^2)^2}-(1-a)\\
&=&\frac{(1-a^2)}{2}\left(\frac{2(1+k)\rho}{1-\rho}+2K(1-a^2)\frac{\rho^2}{(1-\rho^2)^2}-\frac{2}{1+a}\right)\\
&=&\frac{(1-a^2)}{2}\left(1+\frac{2K(1-a^2)\rho^2}{(1-\rho^2)^2}+\left(\frac{2(1+k)\rho}{1-\rho}-1\right)-\frac{2}{1+a}\right).\eeas
A simple computation shows that
\beas \xi'(\rho)&=&\frac{(1+k)(1-a^2)}{\left(1-\rho\right)^2}+K(1-a^2)^2\frac{2\rho(1+\rho^2)}{(1-\rho^2)^3}>0.\eeas
Therefore, $\xi(\rho)$ is an increasing function, and hence $\xi(\rho)\leq \xi(\rho_0)$ for $\rho\leq \rho_0=1/(2k+3)$, where $k\in[0,1)$.
We note that
\beas \xi(\rho_0)=\frac{(1-a^2)}{2}\left(1+\frac{2K(1-a^2)(2k+3)^2}{(2k+4)^2(2k+2)^2}-\frac{2}{1+a}\right).\eeas
Let \beas\Phi(x)=1+\frac{2K(1-x^2)(2k+3)^2}{(2k+4)^2(2k+2)^2}-\frac{2}{1+x},\;x\in[0,1].\eeas 
It is evident that 
\beas&& \Phi(0)=\frac{2K(2k+3)^2}{(2k+4)^2(2k+2)^2}-1,\;\lim_{x\to 1^{-}}\Phi(x)=0\\\text{and}
&&\Phi'(x)=-\frac{4K(2k+3)^2x}{(2k+4)^2(2k+2)^2}+\frac{2}{(1+x)^2}\\
&&=\frac{2}{(1+x)^2}\left(1-\frac{2K(2k+3)^2}{(2k+4)^2(2k+2)^2}x(1+x)^2\right).\eeas
As $x\in[0,1]$, we have 
\beas\Phi'(x)\geq \frac{2}{(1+x)^2}\left(1-\frac{8K(2k+3)^2}{(2k+4)^2(2k+2)^2}\right)\geq 0\;\text{if}\;K\leq \frac{(2k+4)^2(2k+2)^2}{8(2k+3)^2}.\eeas 
Therefore, $\Phi(x)$ is an increasing function on $[0,1]$ for $K\leq(2k+4)^2(2k+2)^2/(8(2k+3)^2)$. Hence, $\Phi(x)\leq 0$ for $x\in[0,1]$, $k\in[0,1)$ and $K\leq (2k+4)^2(2k+2)^2/(8(2k+3)^2)$.\\[2mm]
\indent In order to prove the sharpness of the result, we consider the function $f_3(z)=h_3(z)+\ol{g_3(z)}$ in $\Omega_\gamma$ with $h_3=\psi\circ\Phi_3$ maps $\Omega_\gamma$ univalently 
on $\mathbb{D}$, where $\Phi_3 : \Omega_{\gamma}\to\mathbb{D}$ defined by $\Phi_3(z)=\gamma+(1-\gamma)z$ and $\psi : \mathbb{D}\to\mathbb{D}$ defined by 
$\psi(z)=(a-z)/(1-az)$ for $a\in(0,1)$ and $\gamma\in[0,1)$. For $z\in\Omega_\gamma$, we have
\beas h_3(z)=\psi\left(\gamma+(1-\gamma)z\right)
=\frac{a-(1-\gamma)\left(z+\frac{\gamma}{1-\gamma}\right)}{1-a(1-\gamma)\left(z+\frac{\gamma}{1-\gamma}\right)}
=A_0-\sum_{n=1}^\infty A_n \left(z+\frac{\gamma}{1-\gamma}\right)^n,\eeas	
where $A_0=a$ and $A_n=a^{n-1}(1-a^2)(1-\gamma)^n$ and $g_3(z)=-k\sum_{n=1}^\infty A_n\left(z+\frac{\gamma}{1-\gamma}\right)^n$.
Thus,  
\bs\beas S:&=&\sum_{n=0}^\infty \frac{|a_n|}{(1-\gamma)^n}\rho^n+\sum_{n=1}^\infty\frac{|b_n|}{(1-\gamma)^n}\rho^n+\frac{(2k+4)^2(2k+2)^2}{8(2k+3)^2}\frac{S_\rho^\gamma(h)}{\pi}\\
&=&A_0+\frac{(1-a^2)(1+k)}{a}\sum_{n=1}^\infty\left(a\rho\right)^n+\frac{(2k+4)^2(2k+2)^2}{8(2k+3)^2}\sum_{n=1}^\infty n|A_n|^2\left(\frac{\rho}{1-\gamma}\right)^{2n}\\
&=&1+(1-a)F_3(a, \rho),\eeas\es
where 
\beas F_3(a, \rho)=\frac{(1+a)(1+k)\rho}{1-a\rho}+\frac{(2k+4)^2(2k+2)^2}{8(2k+3)^2}\frac{(1-a)(1+a)^2\rho^2}{(1-a^2\rho^2)^2}-1.\eeas
Differentiating partially $F_3(a,\rho)$ with respect to $\rho$, we have
\beas \frac{\pa}{\pa \rho} F_3(a,\rho)
=\frac{(1+a)(1+k)}{(1-a\rho)^2}+\frac{(2k+4)^2(2k+2)^2}{4(2k+3)^2}\frac{(1-a)(1+a)^2\rho(1+a^2\rho^2)}{(1-a^2\rho^2)^3}> 0\eeas
 for $\rho\in(0,1)$. Therefore, $F_3(a,\rho)$ is a strictly increasing function of $\rho$ in $(0,1)$. Thus, for $\rho>\rho_0=1/(2k+3)$, we have 
\beas F_3(a,\rho)&>&F_3(a,\rho_0)\\&=&\frac{(1+a)(1+k)}{3+2k-a}+\frac{(2k+4)^2(2k+2)^2}{8(2k+3)^2}\frac{(1-a)(1+a)^2(3+2k)^2}{(3+2k+a)^2(3+2k-a)^2}-1.\eeas
It is evident that $F_3(a,\rho_0)\to 0$ as $a\to 1$. Hence, $1+(1-a)F_3(a,\rho)>1$ for $\rho>\rho_0$. 												
This shows that $\rho_0$ is the best possible. This completes the proof.
\end{proof}
\begin{proof}[\bf{Proof of Theorem \ref{th3}}] Let $h(z)$ be analytic on $\Omega_{\gamma}$ with $|h(z)|\leq 1$. In view of \textrm{Lemma \ref{le2}}, we have $|a_n|\leq (1-\gamma)^n (1-|a_0|^2)$ for $n\geq 1$. Thus,
\bea\label{a5}\sum_{n=0}^\infty \frac{|a_n|}{(1-\gamma)^n}\rho^n\leq |a_0|+\frac{(1-|a_0|^2)\rho}{1-\rho}.\eea
By a similar argument as in the proof of \textrm{Theorem \ref{th1}}, \textrm{Lemma \ref{le4}} and $|g'(z)|\leq k|h'(z)|$, we have 
\bea\label{a6} \sum_{n=1}^\infty \frac{|b_n|^2}{(1-\gamma)^n}\rho^n\leq \frac{k(1-|a_0|^2)\rho}{(1-\rho)}\;\;\text{for}\;\left|z(1-\gamma)+\gamma\right|=\rho.\eea
In view of \textrm{Lemma \ref{le4}} and the condition $|g'(z)|\leq k|h'(z)|$ gives that for $\left|(1-\gamma)z+\gamma\right|=\rho$, 
\bea\label{a7}&& \sum_{n=1}^\infty n\frac{|b_n|^2}{(1-\gamma)^{2n}}\rho^{2n}\leq k^2 \sum_{n=1}^\infty n\frac{|a_n|^2}{(1-\gamma)^{2n}}\rho^{2n}\\[2mm]\text{and}
&&\sum_{n=1}^\infty n\frac{|b_n|}{(1-\gamma)^n}\rho^{2n}\leq\left(\sum_{n=1}^\infty n\frac{|b_n|^2}{(1-\gamma)^{2n}}\rho^{2n}\right)^{1/2}\left( \sum_{n=1}^\infty n\rho^{2n}\right)^{1/2}\nonumber\\[2mm]
\label{a8}&\leq& \left(k^2\sum_{n=1}^\infty n\frac{|a_n|^2}{(1-\gamma)^{2n}}\rho^{2n}\right)^{1/2}\left( \frac{\rho^2}{(1-\rho^2)^2}\right)^{1/2}
\leq \frac{k\left(1-|a_0|^2\right)\rho^2}{(1-\rho^2)^2}.\eea
It is well known that, for analytic functions $\varphi_1(z)=\sum_{n=1}^\infty a_nz^n$ and $\varphi_2(z)=\sum_{n=1}^\infty b_nz^n$ with $\phi(z)=\varphi_1(z)+\ol{\varphi_2(z)}$, the 
area $S_\rho$ of the image $\phi\left(|z|<\rho\right)$ (see \cite[Chapter 7]{20P}) is given by
\beas\frac{S_\rho}{\pi}&=&\iint\limits_{\substack{|z|<\rho}}\left(|\varphi_1'(z)|^2-|\varphi_2'(z)|^2\right)dx dy=\sum_{n=1}^\infty n\left(|a_n|^2-|b_n|^2\right)\rho^{2n}.\eeas
Since $z\in\Omega_{\gamma}$, so we write $z=(\xi-\gamma)/(1-\gamma)$ for $\xi=\xi_1+i\xi_2\in\mathbb{D}$ and the Jacobian $J_f$ of $f$ is given by 
\beas J_f \left(\frac{\xi-\gamma}{1-\gamma}\right) = \left|h'\left(\frac{\xi-\gamma}{1-\gamma}\right)\right|^2 - \left|g'\left(\frac{\xi-\gamma}{1-\gamma}\right)\right|^2.\eeas Therefore,
\beas h\left(\frac{\xi-\gamma}{1-\gamma}\right)=\sum\limits_{n=0}^\infty \frac{a_n}{(1-\gamma)^n}\xi^n\;\text{and}\;g\left(\frac{\xi-\gamma}{1-\gamma}\right)=\sum\limits_{n=0}^\infty \frac{b_n}{(1-\gamma)^n}\xi^n\;\text{for}\;\xi\in\mathbb{D}.\eeas
Thus, the area of the image $f\left(|\xi|<\rho\right)$, {\it i.e.}, $f\left(|z+\gamma/(1-\gamma)|<\rho/(1-\gamma)\right)$ is given by
\bea\label{a9}\frac{S_\rho^\gamma(f)}{\pi}&=&\frac{1}{(1-\gamma)^2}\iint\limits_{\substack{|\xi|<\rho}}\left(\left|h'\left(\frac{\xi-\gamma}{1-\gamma}\right)\right|^2-\left|g'\left(\frac{\xi-\gamma}{1-\gamma}\right)\right|^2\right)d\xi_1 d\xi_2\nonumber\\[2mm]
&=&\frac{1}{(1-\gamma)^2}\sum_{n=1}^\infty n\left(\frac{|a_n|^2}{(1-\gamma)^{2n}}-\frac{|b_n|^2}{(1-\gamma)^{2n}}\right)\rho^{2n}\nonumber\\[2mm]
&=&\frac{1}{(1-\gamma)^2}\sum_{n=1}^\infty n\left(\frac{|a_n|}{(1-\gamma)^n}+\frac{|b_n|}{(1-\gamma)^n}\right)\left(\frac{|a_n|}{(1-\gamma)^n}-\frac{|b_n|}{(1-\gamma)^n}\right)\rho^{2n}\nonumber\\[2mm]
&\leq&\frac{1}{(1-\gamma)^2} \sum_{n=1}^\infty n\frac{|a_n|}{(1-\gamma)^n}\left(\frac{|a_n|}{(1-\gamma)^n}+\frac{|b_n|}{(1-\gamma)^n}\right)\rho^{2n}\nonumber\\[2mm]
&\leq&\frac{(1-|a_0|^2)^2}{(1-\gamma)^2}\left(1+k\right)\sum_{n=1}^\infty n\rho^{2n}=\frac{(1+k)(1-|a_0|^2)^2}{(1-\gamma)^2}\frac{\rho^2}{(1-\rho^2)^2}.\eea
Let $|a_0|=a\in[0,1]$. From (\ref{a5}), (\ref{a6}) and (\ref{a9}), we have 
\beas &&\sum_{n=0}^\infty|a_n|\rho^n+\sum_{n=1}^\infty|b_n|\rho^n+K\left(\frac{S_\rho^\gamma(f)}{\pi}\right)\\
&&\leq a+(1+k)\frac{(1-a^2)\rho}{1-\rho}+K\frac{(1+k)(1-a^2)^2}{(1-\gamma)^2}\frac{\rho^2}{(1-\rho^2)^2}\leq 1+\Phi(\rho),\eeas
where 
\beas &&\Phi(\rho)=(1+k)\frac{(1-a^2)\rho}{1-\rho}+K\frac{(1+k)(1-a^2)^2}{(1-\gamma)^2}\frac{\rho^2}{(1-\rho^2)^2}-(1-a)\\
&&=\frac{1-a^2}{2}\left(1+\frac{2K(1+k)\left(1-a^2\right)}{(1-\gamma)^2}\frac{\rho^2}{(1-\rho^2)^2}+\left(\frac{2(1+k)\rho}{1-\rho}-1\right)-\frac{2}{1+a}\right).\eeas
Note that 
\beas\Phi'(\rho)= (1+k)\frac{(1-a^2)}{(1-\rho)^2}+K\frac{(1+k)\left(1-a^2\right)^2}{(1-\gamma)^2}\frac{2\rho(1+\rho^2)}{(1-\rho^2)^3}>0.\eeas
Therefore, $\Phi(\rho)$ is an increasing function and $\Phi(\rho)\leq \Phi(\rho_0)$ for $\rho\leq \rho_0=1/(2k+3)$.
To prove that $\Phi(\rho)\leq 0$, it suffices to show that $\Phi(\rho_0)\leq 0$ for all $a\leq 1$. Now
\beas \Phi(\rho_0)=\frac{1-a^2}{2}\left(1+\frac{2K(1+k)\left(1-a^2\right)}{(1-\gamma)^2}\frac{(2k+3)^2}{(2k+4)^2(2k+2)^2}-\frac{2}{(1+a)}\right).\eeas
Let 
\beas F(x)=1+\frac{2K(1+k)(2k+3)^2}{(1-\gamma)^2(2k+4)^2(2k+2)^2}\left(1-x^2\right)-\frac{2}{1+x}\quad\text{for}\; x\in[0,1].\eeas
It is enough to show that $F(x)\leq 0$ for $x\in[0,1]$, so that $\Phi(\rho_0)\leq 0$. It is easy to see that
\beas F(0)=\frac{2K(1+k)(2k+3)^2}{(1-\gamma)^2(2k+4)^2(2k+2)^2}-1\;\text{and}\;\lim_{x\to 1^{-}}F(x)=0.\eeas
By differentiating $F(x)$ with respect to $x$, we have
\beas &&F'(x)=-\frac{4K(1+k)(2k+3)^2x}{(1-\gamma)^2(2k+4)^2(2k+2)^2}+\frac{2}{(1+x)^2}\\[2mm]
&&=\frac{2}{(1+x)^2}\left(1-\frac{2K(1+k)(2k+3)^2}{(1-\gamma)^2(2k+4)^2(2k+2)^2}x(1+x)^2\right)\\[2mm]
&&\geq\frac{2}{(1+x)^2}\left(1-\frac{8K(1+k)(2k+3)^2}{(2k+4)^2(2k+2)^2}\right)\geq 0\;\text{if}\; K\leq \frac{(2k+4)^2(2k+2)^2}{8(1+k)(2k+3)^2}.\eeas
Therefore, $F(x)$ is a monotonically increasing function of $x\in[0,1]$ if $K\leq (2k+4)^2(2k+2)^2/(8(1+k)(2k+3)^2)$. Hence $F(x)\leq 0$ for $x\in[0,1]$. \\[2mm]
\indent To prove the sharpness of the result, we consider $f_4(z)=h_4(z)+\ol{g_4(z)}$ in $\Omega_{\gamma}$, where $h_4=\psi\circ\Phi_4$ with $\Phi_4 : \Omega_{\gamma}\to\mathbb{D}$ defined by $\Phi_4(z)=\gamma+(1-\gamma)z$ and $\psi : \mathbb{D}\to\mathbb{D}$ defined by $\psi(z)=(a-z)/(1-az)$ for $a\in(0,1)$ and $\gamma\in[0,1)$. Thus
\beas&& h_4(z)=\psi\left(\gamma+(1-\gamma)z\right)
=\frac{a-(1-\gamma)\left(z+\frac{\gamma}{1-\gamma}\right)}{1-a(1-\gamma)\left(z+\frac{\gamma}{1-\gamma}\right)}\\
&&=A_0-\sum_{n=1}^\infty A_n \left(z+\frac{\gamma}{1-\gamma}\right)^n\;\text{for}\;z\in\Omega_\gamma,\eeas	
where $A_0=a$, $A_n=a^{n-1}(1-a^2)(1-\gamma)^n$ and $g_4(z)=-k\sum_{n=1}^\infty A_n\left(z+\frac{\gamma}{1-\gamma}\right)^n$.
Therefore
\bs\beas S:&=&\sum_{n=0}^\infty \frac{|a_n|}{(1-\gamma)^n}\rho^n+\sum_{n=1}^\infty\frac{|b_n|}{(1-\gamma)^n}\rho^n+\frac{(2k+4)^2(2k+2)^2}{8(2k+3)^2}\frac{S_\rho^\gamma(h)}{\pi}\\
&=&A_0+\frac{(1-a^2)(1+k)}{a}\sum_{n=1}^\infty\left(a\rho\right)^n\\
&&+\frac{(2k+4)^2(2k+2)^2}{8(2k+3)^2}\frac{1}{(1-\gamma)^2}\sum_{n=1}^\infty n\left(\frac{|A_n|^2}{(1-\gamma)^{2n}}-\frac{k^2|A_n|^2}{(1-\gamma)^{2n}}\right)\rho^{2n}\\
&=&1+(1-a)F_4(\rho),\eeas\es
where 
\beas F_4(\rho)=\frac{(1+a)(1+k)\rho}{1-a\rho}+\frac{(2k+4)^2(2k+2)^2(1-k^2)}{8(2k+3)^2(1-\gamma)^2}\frac{(1-a)(1+a)^2 \rho^2}{(1-a^2\rho^2)^2}-1.\eeas
Then, we have  
\beas F_4'(\rho)&=&\frac{(1+a)(1+k)}{(1-a\rho)^2}+\frac{(2k+4)^2(2k+2)^2(1-k^2)}{8(2k+3)^2(1-\gamma)^2}\frac{2(1-a)(1+a)^2\rho(1+a^2\rho^2)}{(1-a^2\rho^2)^3}\\
&=&\frac{(1+a)(1+k)}{(1-a\rho)^2}+\frac{(2k+4)^2(2k+2)^2(1-k^2)(1-a)(1+a)^2\rho(1+a^2\rho^2)}{4(2k+3)^2(1-\gamma)^2(1-a^2\rho^2)^3}> 0\eeas
for $\rho\in(0,1)$. Therefore $F_4(\rho)$ is a strictly increasing function of $\rho$ in $(0,1)$. Thus, for $\rho>\rho_0=1/(2k+3)$, we have 
\beas F_4(\rho)>F_4(\rho_0)=\frac{(1+a)(1+k)}{3+2k-a}+\frac{(2k+4)^2(2k+2)^2(1-k^2)(1-a)(1+a)^2}{8(1-\gamma)^2(3+2k+a)^2(3+2k-a)^2}-1.\eeas
Clearly, $F_4(\rho_0)\to 0$ as $a\to 1$. Hence, $1+(1-a)F_4(\rho)>1$ for $\rho>\rho_0$. 												
This shows that $\rho_0$ is best possible. This completes the proof. \end{proof}
\section{Declarations}
\noindent{\bf Acknowledgment:} The work of the second author is supported by the University Grants Commission (IN) fellowship (No. F. 44- 1/2018 (SA- III)).\\[1mm]
{\bf Conflict of Interest:} The authors declare that there are no conflicts of interest regarding the publication of this paper.\\[1mm]
{\bf Availability of data and materials:} Data sharing is not applicable to this article as no data sets were generated or analyzed during the current study.

\end{document}